\documentclass{article}
\usepackage[T1]{fontenc}
\usepackage{amsthm, fullpage}
\usepackage{amsmath,enumerate}
\usepackage{amssymb}
\usepackage{amsfonts}
\usepackage{eucal}
\usepackage[all]{xy}
\usepackage{pslatex}
\usepackage{hyperref}
\providecommand{\scr}{\mathcal}
\usepackage{mathrsfs}

\newtheorem{prop}{Proposition}[section]

\newtheorem{coro}[prop]{Corollary}
\newtheorem{lemm}[prop]{Lemma}
\newtheorem{lem}[prop]{Lemma}
\newtheorem*{lemm*}{Lemma}

\theoremstyle{definition}

\newtheorem{empt}[prop]{}
\newtheorem{dfn}[prop]{Definition}
\newtheorem{rem}[prop]{Remark} 
\newtheorem{ntn}[prop]{Notation} 
\newtheorem{ex}[prop]{Examples} 
\newtheorem*{rem*}{Remark}

\theoremstyle{thm}
\newtheorem{thm}[prop]{Theorem}
\newtheorem*{thm*}{Theorem}
\newtheorem*{lem*}{Lemma}

\newtheorem*{cor*}{Corollary}
\newtheorem*{prop*}{Proposition}

\theoremstyle{dfn}
\newtheorem*{dfn*}{Definition}

\DeclareMathOperator{\coker}{Coker}
\DeclareMathOperator{\im}{Im}
\DeclareMathOperator{\coim}{Coim}
\DeclareMathOperator{\sym}{Sym}

\numberwithin{equation}{prop}

\newcommand{\riso}{ \overset{\sim}{\longrightarrow}\, }

\newcommand{\Spec}{\mathrm{Spec}\,}

\newcommand{\gr}{\mathrm{gr}}

\newcommand{\Ext}{\mathrm{Ext}}

\newcommand{\FF}{{\mathcal{F}}}

\newcommand{\E}{{\mathcal{E}}}
\newcommand{\G}{{\mathcal{G}}}
\renewcommand{\H}{{\mathcal{H}}}
\newcommand{\M}{{\mathcal{M}}}
\newcommand{\NN}{{\mathcal{N}}}
\newcommand{\D}{{\mathcal{D}}}

\renewcommand{\O}{{\mathcal{O}}}

\newcommand{\V}{\mathcal{V}}
\renewcommand{\S}{\mathcal{S}}

\newcommand{\ZZ}{\mathcal{Z}}
\newcommand{\X}{\mathfrak{X}}

\renewcommand{\L}{\mathbb{L}}
\newcommand{\R}{\mathbb{R}}
\newcommand{\Q}{\mathbb{Q}}
\newcommand{\Z}{\mathbb{Z}}
\newcommand{\N}{\mathbb{N}}

\newcommand{\hdag}{  \phantom{}{^{\dag} }    }

\begin{document}

\title{Lagrangianity for log extendable overconvergent  $F$-isocrystals}
\author{Daniel Caro} 
\date{}

\maketitle

\begin{abstract}
In the framework of Berthelot's theory of arithmetic $\D$-modules, 
we prove that Berthelot's characteristic variety associated with a holonomic $\D$-modules endowed with a Frobenius structure
has pure dimension. As an application, we get the lagrangianity of the characteristic variety of
a log extendable overconvergent  $F$-isocrystal.

\end{abstract}

\tableofcontents

\bigskip

\section*{Introduction}
Let $\V$ be a complete discrete valued ring of mixed characteristic $(0,p)$, $\pi$ be a uniformizer, 
$K$ its field of fractions,  
$k$ its residue field which is supposed to be perfect. 
Let $\X$ be a smooth formal $\V$-scheme (the topology is the $p$-adic one),
$X$ be its special fiber. Berthelot has built the sheaf differential operator with finite level over $\X$
which he denotes by $\D ^\dag _{\X}$ (this corresponds somehow to the weak $p$-adic completion of the usual 
sheaf of differential operators). 
Putting $\D ^\dag _{\X,\Q } := \D ^\dag _{\X } \otimes _{\Z} \Q $,
he has defined in \cite[5.2.7]{Beintro2} the characteristic variety $\mathrm{Car} (\E)$ (included in the cotangent space of $X$) 
associated with 
a coherent $\D ^\dag _{\X, \Q }  $-module $\E$ endowed with a Frobenius structure.
He has proved Bernstein's inequality $\dim \mathrm{Car} (\E) \geq \dim X$ and has defined 
$\E$ to be holonomic when this inequality is an equality. 
In this paper, we first prove that when $\E$ is holonomic, 
its characteristic variety $\mathrm{Car} (\E) $ is of pure dimension $\dim X$. 
One main ingredient of the proof is to use the homological characterization of the holonomicity (see \cite[III.4.2]{virrion})
and another one is to use the sheaf of microdifferential operators (for instance, see \cite{Abe-microdiff}).
Both ideas comes from  the original proof of Kashiwara of the analogous property 
in the theory of analytic $\D$-modules (see \cite{Kashiwara-B-function-Hol}).
Finally, when $\E$ is a log extendable overconvergent $F$-isocrystal, 
we establish the inclusion of $\mathrm{Car} (\E) $ into a explicit lagrangian subvariety of the cotangent space of $X$.
With the above purity theorem, this inclusion implies the Lagrangianity of $\mathrm{Car} (\E) $.
Moreover, one another application of this inclusion in a further work will be to get 
some ``relative generic $\O$-coherence'' 
(see precisely the proof of Theorem \cite[1.4.3]{Caro-BettiNumbers}).
This will imply some Betti number estimates (see \cite{Caro-BettiNumbers}).
In the theory of arithmetic $\D$-modules, we recall that  to check some property 
we are often able to reduce to the case of log extendable overconvergent $F$-isocrystals 
(e.g. in the proof of Theorem \cite[1.4.3]{Caro-BettiNumbers}). 
Indeed, overholonomic $F$-complexes of arithmetic $\D$-modules are devissable in overconvergent $F$-isocrystals 
(see \cite{caro_devissge_surcoh})
and thanks to Kedlaya's semistable reduction theorem
any overconvergent $F$-isocrystal becomes log-extendable after the pull-back by some generically etale alteration 
(see \cite{kedlaya-semistableIV}).

\section*{Convention, notation of the paper}
Let $\V$ be a complete discrete valued ring of mixed characteristic $(0,p)$, $\pi$ be a uniformizer, 
$K$ its field of fractions,  
$k$ its residue field which is supposed to be perfect. 
A $k$-variety is a separated reduced scheme of finite type over $k$.
We will denote formal schemes by curly or gothic letters and the corresponding straight roman letter will
mean the special fiber (e.g. if $\X$ is a formal scheme over $\V$, then $X$ is the $k$-variety equal to the special fiber of $\X$).
When $M$ is a $\V$-module, we denote by 
$\widehat{M}$ its $\pi$-adic completion and we set
$M _\Q := M \otimes _{\V} K$.
By default, a module will mean a left module. 

\section{Convention and preliminaries on filtered modules}

We use here the terminology of Laumon in \cite[A.1]{Laumon-TransfCanon}:
\begin{enumerate}[(i)]
\item A filtered ring 
$( D , D _i)$ is a ring $D$, unitary, non necessary commutative, with an increasing filtration by additive subgroups $(D _i) _{i\in \Z}$
indexed by $\Z$ such that $1 \in D _0$ and $D _i \cdot D _j \subset D _{i+j}$ for any $i,j\in \Z$.

\item Let $( D , D _i)$ be a filtered ring. 
We get an exact (not abelian) category of filtered $( D , D _i)$-modules as follows. 
A filtered  $( D , D _i)$-module $( M , M _i)$ is a $D$-module $M$
endowed with a filtration $(M _i) _{i\in \Z}$ such that 
$A _i \cdot M _j \subset M _{i+j}$. A morphism of 
$( M , M _i)\to ( M ', M '_i)$ of filtered $( D , D _i)$-modules is 
a morphism of $D$-modules $f \colon M \to M'$ such that
$f (M _i ) \subset M' _i$. 
If $( M , M _i)$ is a filtered $( D , D _i)$-module and $n \in \Z$, 
we denote by 
$(M (n), M (n) _{i} )$ the filtered $( D , D _i)$-module defined as follows: 
$M (n) = M$ and $M (n) _{i} := M_{i+n}$.
Following \cite[2.1.2]{EGAII},
a filt free $( D , D _i)$-module (resp. of finite type) is 
a direct sum (resp. a finite direct sum) in the category of filtered $( D , D _i)$-modules
of the form 
$( D ( n ) , D ( n) _i)$, for some integer 
$n$. 
Let $( M , M _i)$ be a filtered $( D , D _i)$-module. 
We say that the filtration $M _{i}$ is good or that
$( M , M _i)$ is a good a filtered $( D , D _i)$-module
if there exists an epimorphism 
in the category of filtered $( D , D _i)$-modules
of the form
$\phi \colon ( L, L _i) \twoheadrightarrow ( M , M _i)$
such that 
$\phi ( L _i) = M _i$.
We remark that any $D$-module of finite type can be endowed with a 
good filtration. Conversely, for any good filtered $( D , D _i)$-module 
$( M , M _i)$, $M$ is a $D$-module of finite type. 

\item 
Let $( D , D _i)$ be a filtered ring and $( M , M _i)$ be a filtered $( D , D _i)$-module. 
The  {\it ind-pro-complete separation} of $( M , M _i)$, denoted by $(\widehat{M}, \widehat{M} _i)$ is a filtered $( D , D _i)$-module
defined as follows:
$ \widehat{M} _i := \underleftarrow{\lim} _{n} M _i /M _{i-n}$ is the complete separation of $M _i$ with respect to the 
filtration $(M _{i-n}) _{n\in\N}$ and 
$ \widehat{M} := \cup _{i\in \Z} \widehat{M} _i$, where
the inclusion $\widehat{M} _i \subset \widehat{M} _{i+1}$ are that induced by complete separation from
the inclusion $M _i \subset M _{i+1}$. 
Using the universal property of projective limits we check that 
$(\widehat{D}, \widehat{D} _i)$ is also a filtered ring and 
that $(\widehat{M}, \widehat{M} _i)$ is a filtered 
$(\widehat{D}, \widehat{D} _i)$-module.

We say that $( M , M _i)$ is {\it ind-pro-complete separated} if 
the canonical morphism 
$( M , M _i) \to (\widehat{M}, \widehat{M} _i)$ is an isomorphism.
For instance, we remark that the filtration of an ind-pro-complete separated filtered ring is exhaustive. 
As Laumon, to simplify the terminology (we hope there will not be confusing with the usual notion of completion),
we will simply say ``complete'' for ``ind-pro-complete separated''
and ``completion'' for `ìnd-pro-complete separation''. 

\item In this section, with our abuse of terminology, 
$( D , D _i)$ will be a complete filtered ring such that $\gr ( D , D _i)$ is a left and right noetherian ring. 
Hence, from Proposition \cite[A.1.1]{Laumon-TransfCanon}, a good filtered $(D, D  _i)$-module is complete.
\end{enumerate}

Lemma \cite[3.3.2]{Laumon-CatFiltDmod} 
is still valid in the following context:
\begin{lem}
\label{Laumon332}
Let $0 \to (M', M' _i) \overset{f}{\longrightarrow} (M, M _i)
\overset{g}{\longrightarrow} (M'', M'' _i) \to 0$
be a sequence of morphisms of good filtered $(D, D  _i)$-modules.
The following conditions are equivalente :
\begin{enumerate}[(a)]
\item we have $M' _i= M' \cap M _i$ and $M'' _i = g (M' _i)$ for any $i\in \Z$; 
\item the sequences of abelian groups 
$0 \to M '_i \overset{f}{\longrightarrow} M _i\overset{g}{\longrightarrow} M''  _i\to 0$ are exact for any $i\in \Z$;
\item $g \circ f = 0$ and the sequence of $\gr D$-modules 
$0 \to \gr M' \to \gr M \to \gr M'' \to 0$ is exact. 
\end{enumerate}
When these equivalent conditions are satisfied, the sequence
of $D$-modules 
$0 \to M' \to M \to M''\to 0$ is exact.

\end{lem}

\begin{proof}
The equivalence between $(a)$ and $(b)$ is obvious. 
Suppose the condition $(b)$ is satisfied. 
The fact  good filtrations are exhaustive implies that $g \circ f =0$. 
We get the last condition $(c)$ by using the nine Lemma 
(see the exercice \cite[1.3.2]{Weibel-HomologicalAlg}).
Suppose now the condition $(c)$ is satisfied. 
Using the nine Lemma (more precisely, the part 3 of the exercice \cite[1.3.2]{Weibel-HomologicalAlg}), 
for any $i\in \Z$, we check by induction in $n\geq 1$ that
the sequence of abelian groups 
$0 \to M' _i /M' _{i-n} \to M _i /M _{i-n}\to M'' _i /M'' _{i-n}\to0$ are exact. 
Taking the projective limits, since $M' _i$, $M _i$ and $M'' _i$ are complete separated by hypothesis, 
since Mittag Leffler condition is satisfied, we get the sequence
$0 \to M '_i \overset{f}{\longrightarrow} M _i\overset{g}{\longrightarrow} M''  _i\to 0$ is exact. 
The last statement of the Lemma follows from the remark that filtrations are exhaustive. 
\end{proof}

\begin{dfn}
Let $0 \to (M', M' _i) \overset{f}{\longrightarrow} (M, M _i)
\overset{g}{\longrightarrow} (M'', M'' _i) \to 0$
be a sequence of morphisms of good filtered $(D, D  _i)$-modules
satisfying conditions (a) and (b) of \ref{Laumon332}.
We say that this sequence is an ``exact'' sequence of morphisms of good filtered $(D, D  _i)$-modules.
\end{dfn}

\begin{lem}
\label{lemexseq-imu-coimu}
Let 
$u\colon (M, M _i) \to (N, N _i)$ be a morphism of good filtered $(D, D  _i)$-modules. 
We have the exact sequences of good filtered $(D, D  _i)$-modules:
\begin{gather}
\notag
0 \to \ker u \to (M, M _i) \to \mathrm{Coim} u \to 0, 
\\
\label{exseq-imu-coimu}
0 \to \mathrm{Im} u \to (N, N _i) \to  \mathrm{Coker} u \to 0.
\end{gather}
\end{lem}

\begin{proof}
Let $\ker u$ be the kernel of $u$ in the category of filtered $(D, D  _i)$-modules, 
i.e. $\ker u = (\ker u, \ker u \cap M _i)$.
Let $\coker u$ be the cokernel of $u$ in the category of filtered $(D, D  _i)$-modules, 
i.e. $\coker u = (\coker u, N _i/N _i \cap u (M))$ where $\phi \colon N \to \coker u$ is the canonical morphism.
From \cite[A.1.1.2]{Laumon-TransfCanon}, the filtered $(D, D  _i)$-modules $\ker u$
and
$ \mathrm{Coker} u $ are good.
Hence, $\ker u, \coker u, \im u, \coim u$ exist in the category of good filtered $(D, D  _i)$-modules
(and are equal to that computed in the category of  good filtered $(D, D  _i)$-modules).
Hence, both sequences 
\ref{exseq-imu-coimu} are well defined in the category of good filtered $(D, D  _i)$-modules. 
Since $\coim u = (\im u, u (M _i))$ and $\im u = (\im u, \im u \cap N _i)$,
then these sequence satisfy the condition (a) of Lemma \ref{Laumon332}
and hence they are exact.

\end{proof}

\begin{dfn}
[Strictness]
\label{strictness}
A morphism 
$u\colon (M, M _i) \to (N, N _i)$ of good filtered $(D, D  _i)$-modules is strict if 
the canonical morphism 
$\mathrm{Coim} u \to \mathrm{Im} u$ is an isomorphism of (good) filtered $(D, D  _i)$-modules.
If  $u\colon M \to N$ is a monomorphism (resp. epimorphism) and if $u\colon (M, M _i) \to (N, N _i)$ is strict, 
we say that $u $ is a strict monomorphism (resp. strict epimorphism). 
\end{dfn}

\begin{lem}
\label{lem-gru-injetc}
Let 
$u\colon (M, M _i) \to (N, N _i)$ be a morphism of good filtered $(D, D  _i)$-modules. 

\begin{enumerate}
\item Then $u$ is strict if and only if $u (M _i) = u (M) \cap N _i$ for any $i \in \Z$. 
\item The following conditions are equivalent
\begin{enumerate}
\item $u$ is a strict monomorphism 
\item  the morphism 
$(M, M _i) \to \mathrm{Im} u$ is an isomorphism
\item 
the sequence of good filtered $(D, D  _i)$-modules
$0 \to (M, M _i) \to (N, N _i) \to  \mathrm{Coker} u \to 0$
is exact. 
\item $\gr u$ is a monomorphism. 
\end{enumerate}

\item The following conditions are equivalent

\begin{enumerate}
\item $u$ is a strict epimorphism
\item  the morphism 
$ \coim u \to (N, N _i)$ is an isomorphism
\item the 
sequence of good filtered $(D, D  _i)$-modules
$0 \to \ker u \to (M, M _i) \to  (N, N _i) \to 0$
is exact. 
\item $\gr u$ is an epimorphism. 
\end{enumerate}

\item $u$ is an isomorphism if and only if $\gr u$ is an isomorphism.
\end{enumerate}
\end{lem}

\begin{proof}
The first statement is straighforward from the description of $\im u$ and $\coim u$.
and from Lemma \ref{lemexseq-imu-coimu}. Let us check 2. 
$(a) \Rightarrow (b)$ is clear from the description of $\im u$ and from $1$. 
The implication $(b) \Rightarrow (c)$ (resp. $(c) \Rightarrow (d)$) is a consequence of \ref{exseq-imu-coimu}.
(resp. \ref{Laumon332}). Finally, suppose $(d)$ is satisfied. Let $x \in \ker u$.
Suppose $x \not = 0$.  
There exists $i \in \Z$ such that $x \not \in M _i$ (recall the filtration is separated).
This is a contradiction with the fact that  
$M _{i+1}/M _i \to N _{i+1}/N _i$ is injective (because $\gr u$ is injective by hypothesis). 
Hence $u $ is a monomorphism.
The fact that $\gr u$ is injective implies that $ M _{i+1}\cap N _i = M _i$, for any $i \in \Z$. By induction in $n \in \N$ we get that 
for any $i \in \Z$, $ M _{i+n}\cap N _i = M _i$. Since the filtration is exhaustive, 
this yields that $M \cap N _i = M _i$. From part 1) of the Lemma, this means that
$u$ is strict. 
Let us check part 3). We check similarly
$(a) \Rightarrow (b) \Rightarrow (c) \Rightarrow (d)$. 
Now suppose that $\gr u$ is an epimorphism. 
Let $i \in \Z$. From part 1), it is sufficient to check 
$u (M _i) = N _i$ (indeed filtrations are exhausted and then $u$ will be surjective).
Let $y \in N _i$.
Put $y _{-1}:= y$.  
By induction in $n \geq 0$, we construct 
$y _n \in N _{i-1-n}$ and $x _n \in M _{i-n}$ such 
that $u (x _n ) = y _{n-1} - y _n$.
This is consequence of the equality 
$u (M _{i -n }) + N _{i -n-1} =N _{i -n }$ (because $\gr u$ is an epimorphism).
Since $N _i$ and $M _i$ are separated complete for the filtrations
$(N _{i-n}) _{n\in \N}$ and $(M _{i-n}) _{n\in \N}$, 
the sum $\sum _{n \geq -1} (y _n - y _{n+1})$ converge to $y$ and $\sum _{n \in \N} x _{n}$ converge in $M _i$ to an element,
denoted by $x$. Hence, $u (x) = y$ and then $N _i \subset u ( M_i)$.  
Finally, $4$ is a consequence of the equivalence $(a) \Leftrightarrow (d)$ of $2$ and $3$. 
\end{proof}

\begin{rem}
\label{strictness-rem-comp}
With the notation \ref{strictness-comp}, 
this is not true in general that if $u$ and $v$ are strict then $v \circ u$ is also strict. 
However, using \ref{strictness-comp} and \ref{lemexseq-imu-coimu}, 
we remark that a morphism of good filtered $(D, D  _i)$-modules
is strict if and only if it is the composition (in the category of good filtered $(D, D  _i)$-modules)
of a strict epimorphism with a strict monomorphism.
This last characterization of strictness was Laumon's definition of strictness given in 
\cite[1.0.1]{Laumon-CatFiltDmod}.
\end{rem}

\begin{rem}
\label{rem-strict-epi-good}
Let $(M, M _i)$ be a filtered $(D, D  _i)$-module. 
From \ref{lem-gru-injetc}.1, we remark that 
$(M, M _i)$ is a good filtered $(D, D  _i)$-module if and only if there exists
a strict epimorphism of the form 
$u\colon (L, L _i) \twoheadrightarrow (M, M _i)$, where
$(L, L _i)$ is a free filtered $(D, D  _i)$-module of finite type.
\end{rem}

\begin{lem}
Let 
$u\colon (M, M _i) \to (N, N _i)$ be a morphism of good filtered $(D, D  _i)$-modules. 

\begin{enumerate}
\item The following assertions are equivalent :
	\begin{enumerate}
	\item The morphism $u$ is strict ; 
	\item The sequence 
	$0 \to \gr \ker u \to \gr (M, M _i) \to \gr (N, N _i) \to  \gr \mathrm{Coker} u \to 0$
	is exact. 
	\item $\ker \gr (u) = \gr \ker (u)$ and $\mathrm{coker} \,\gr (u) = \gr \, \mathrm{coker} (u)$.
	\end{enumerate}
\item If $u$ is strict then we have also 
$\mathrm{im} \,\gr (u) = \gr \, \mathrm{im} (u)$.
\end{enumerate}

\end{lem}

\begin{proof}
By applying the functor
$\gr $ to the exact sequences \ref{exseq-imu-coimu}, we get that 
$(a) \Rightarrow (b)$.
Conversely, suppose $(b)$ satisfied. 
First, remark the following fact available in an abelian category $\mathfrak{A}$: 
let 
$\alpha \colon M _1 \to M _2$ 
(resp. $\alpha \colon M _2 \to M _3$, resp. $\alpha \colon M _3 \to M _4$) 
be a epimorphism (resp. a morphism, resp. a monomorphism) 
of $\mathfrak{A}$. Then if $\ker \alpha = \ker \gamma \circ \beta \circ \alpha$
then $\beta$ is a monomorphism. 
Moreover, if $\im \gamma = \im \gamma \circ \beta \circ \alpha$,
then $\beta$ is surjective. 
By applying the functor
$\gr $ to the exact sequences \ref{exseq-imu-coimu},
with this remark, the condition $(b)$ implies that 
the morphism 
$\gr \coim (u) \to \gr \im (u)$ is an isomorphism
of the abelian category of $\gr D$-modules. 
With Lemma \ref{lem-gru-injetc}.4, 
this implies that $\coim (u) \to \im (u) $ is an isomorphism.

The equivalence $(b) \Leftrightarrow (c)$ is straightforward.
We check the statement 2) by applying $\gr $ to the exact sequences \ref{exseq-imu-coimu}.
\end{proof}

\begin{lem}
\label{strictness-comp}
Let 
$u\colon (M, M _i) \to (N, N _i)$ 
and
$v\colon (N, N _i) \to ( O, O _i)$
be two morphisms of good filtered $(D, D  _i)$-modules. 

\begin{enumerate}
\item If $v$ is a strict monomorphism and $u$ is strict
then $v \circ u$ is strict. 
\item If $u$ is  a strict epimorphism and $v$ is strict
then $v \circ u$ is strict. 
\item If $v \circ u$ is strict epimorphism then $v$ is a strict epimorphism. 
\item If $v \circ u$ is strict monomorphism then $u$ is a strict monomorphism. 
\end{enumerate}
\end{lem}

\begin{proof}
This can be checked elementarily from 
the characterization \ref{lem-gru-injetc}.1 
For instance, let us check 1.
Suppose $v$ is a strict monomorphism and $u$ is strict.
We have  
$v ( u ( M )) \cap O _i 
\subset 
v (N) \cap O _i 
= v (N _i)$ 
(because $v$ is strict).
Hence, we get 
$v ( u ( M )) \cap O _i 
\subset
v ( u ( M )) \cap v ( N _i) 
=
v ( u ( M ) \cap  N _i) 
=
v ( u ( M _i)) $.
This implies $v ( u ( M _i))
=v ( u ( M )) \cap O _i $
(use that $v$ is a monomorphism and $u$ is strict for the equalities).
Let us check 4. 
If $v \circ u$ is strict monomorphism then $u$ is a monomorphism
and we have
$M _i  \subset M \cap u ^{-1} (N _i) \subset M \cap (v \circ u ) ^{-1} ( O _i) = M _i$. 
Hence, $M _i = M \cap u ^{-1} (N _i)$, i.e. $u (M _i) = u (M) \cap N _i$.
We leave the other statements to the reader. 
\end{proof}

\begin{prop}
\label{exact-category}
With this definition of strictness, 
the category of good filtered $(D, D  _i)$-modules is exact (see the definition in \cite[1.0.2]{Laumon-CatFiltDmod}).
\end{prop}

\begin{proof}
This is straighforward from previous Lemmas. 
For instance, the condition \cite[1.0.2.(vi)]{Laumon-CatFiltDmod} are the last two statements of 
\ref{strictness-comp}.
\end{proof}

\begin{ntn}
[Localisation]
\label{nota-local}
Let $f$ be a homogeneous element of $\gr D$.
We denote by $(D _{[f]}, D _{[f],i})$ the complete filtered ring of $(D , D _{i})$ relatively to 
$S _{1} (f):= \{ f ^{n}\; ,\; n \in \N\}\subset \gr D$ (see the definition after \cite[Corollaire A.2.3.4]{Laumon-TransfCanon}).

Let $(M, M _i)$ be a good filtered $(D, D  _i)$-module.
We put 
\begin{equation}
\label{nota-local-defM}
(M _{[f]}, M _{[f],i}): = (D _{[f]}, D _{[f],i}) \otimes _{ (D, D _{i}) } (M, M _{i}),
\end{equation}
the localized filtered module of $(M, M _i)$ with respect to $S _{1} (f)$.
We remind that 
$(M _{[f]}, M _{[f],i})$ is also a good filtered 
$(D _{[f]}, D _{[f],i})$-module (see \cite[A.2.3.6]{Laumon-TransfCanon})
and $\gr M _{[f]} \riso \gr D _{[f]} \otimes _{\gr D } \gr M$ (see \cite[A.1.1.3]{Laumon-TransfCanon}).
\end{ntn}

The results and proofs of Malgrange in \cite[IV.4.2.3]{Malgrange-Carac-Homol-dim}
(we can also find the proof in the book \cite[D.2.2]{HTT-DmodPervSheavRep})
can be extended without further problem in the context of complete filtered rings:
\begin{lem}
\label{4.2.3.2Malgrange}
Let $(M, M _i)$ be a good filtered $(D, D  _i)$-module.
Then there exists some free filtered $(D, D  _i)$-modules of finite type 
$(L _{n}, L  _{n, i}) $ with 
$n \in \N$ and strict
morphisms of good filtered $(D, D  _i)$-modules 
$(L _{n+1}, L  _{n+1, i}) \to (L _{n}, L  _{n, i}) $
and 
$(L _{0}, L  _{0, i}) 
\to 
(M, M _i)$
such that 
$L _{\bullet} \to M$ is a resolution of $M$ (in the category of $D$-modules).

We call such a resolution
$(L _{\bullet}, L  _{\bullet, i}) $ a ``good resolution'' 
of $(M, M _i)$.

\end{lem}

\begin{proof}
This is almost the same as \cite[IV.4.2.3.2]{Malgrange-Carac-Homol-dim}.
For the reader, we remind the construction: 
with the remark \ref{rem-strict-epi-good}, 
there exists a strict epimorphism of good filtered $(D, D  _i)$-modules
of the form 
$\phi _0 \colon (L _0, L  _{0, i}) \to (M, M _i)$, 
with $(L _0, L  _{0, i})$
a free filtered $(D, D  _i)$-module of finite type.
Let $ (M _1, M _{1,i}) $ be the kernel of $\phi _0$ (in the category of 
good filtered $(D, D  _i)$-modules: see \ref{exact-category}). 
Since $ (M _1, M _{1,i}) $ is good, 
there exists a strict epimorphism 
of the form 
$\phi _1 \colon (L _1, L  _{1, i}) \to (M _1, M _{1,i})$, 
with $(L _1, L  _{1, i})$
a free filtered $(D, D  _i)$-module of finite type.
Hence, the morphism 
$(L _1, L  _{1, i}) \to (L _0, L  _{0, i})$ is strict. 
We go on similarly.

\end{proof}

\begin{rem}
Let $(L _{\bullet}, L  _{\bullet, i}) $ be a good resolution 
of $(M, M _i)$.
Then 
$ \gr (L _{\bullet}, L  _{\bullet, i})$ 
is a resolution of 
$\gr (M, M _i)$ 
by 
free $\gr (D, D  _i)$-modules of finite type (use the properties of strictness given in \ref{exact-category}).
\end{rem}

\begin{lem}
\label{lem-subquotientHgr}
Let $K ^{\bullet}$ be a complex of 
abelian groups. 
Let $(F _{i} K ^{\bullet}) _{i\in \Z}$ be an increasing filtration of $K ^{\bullet}$.
We put 
\begin{equation}
\label{lem-subquotientHgrdef}
F _{i} H ^{r} (K ^{\bullet}): = \mathrm{Im} ( H ^{r} ( F _{i} K ^{\bullet}) \to H ^{r} (K ^{\bullet})).
\end{equation}
Then 
$\gr ^F  _{i} (H ^{r} (K ^{\bullet}))$ is a subquotient of 
$H ^{r} (\gr ^F  _{i} K ^{\bullet})$.
\end{lem}

\begin{proof}
For instance, we can follow the last seven lines of the proof of \cite[D.2.4]{HTT-DmodPervSheavRep}
(or also at Malgrange's description of the corresponding spectral sequence
in \cite[IV.4.2.3.2]{Malgrange-Carac-Homol-dim}): 
denote by $d ^r \colon K ^r \to K ^{r+1}$ the morphism in $K ^\bullet$,
$d _i ^r \colon F _{i}  K ^r \to F _{i} K ^{r+1}$ the morphism in $F _iK ^\bullet$,
$\smash{\overline{d}} _i ^r 
\colon 
\gr ^F _i K ^{r}
=
F _{i}  K ^r / F _{i-1}  K ^r
\to 
F _{i} K ^{r+1}/F _{i-1} K ^{r+1}
=
\gr ^F _i K ^{r+1}$.
Since
$\ker d _i ^r = \ker d ^r  \cap F _i K ^r$, 
we get 
$F _{i} H ^{r} (K ^{\bullet})= \ker d ^r  \cap F _i K ^r + \im d ^{r-1} /\im d ^{r-1} $.
By definition we obtain:
\begin{equation}
\label{lem-subquotientHgrdefpre1}
\gr ^F  _{i} (H ^{r} (K ^{\bullet}))
:=
F _{i} H ^{r} (K ^{\bullet})/F _{i-1} H ^{r} (K ^{\bullet})
=
\ker d ^r  \cap F _i K ^r + \im d ^{r-1} /\ker d ^r  \cap F _{i-1} K ^r + \im d ^{r-1}.
\end{equation}
We have
$\ker \smash{\overline{d}} _i ^r 
=
\ker ( F _{i} K ^{r} \to 
\gr ^F _i K ^{r+1})/F _{i-1} K ^{r}$
and 
$\im \smash{\overline{d}} _{i-1} ^r 
=
d _i ^{r-1} (F _i K ^{r-1}) + F _{i-1} K ^{r} / F _{i-1} K ^{r}$.
Hence
\begin{equation}
\label{lem-subquotientHgrdefpre2}
H ^{r} (\gr ^F  _{i} K ^{\bullet})
:=
\ker \smash{\overline{d}} _i ^r /\im \smash{\overline{d}} _{i-1} ^r 
=
\ker ( F _{i} K ^{r} \to \gr ^F _i K ^{r+1})
/
d _i ^{r-1} (F _i K ^{r-1}) + F _{i-1} K ^{r} .
\end{equation}
Set 
$L= 
\ker d ^r  \cap F _i K ^r /d _i ^{r-1} (F _i K ^{r-1}) + \ker d ^r  \cap F _{i-1} K ^r $.
The inclusion 
$\ker d ^r  \cap F _i K ^r \subset 
\ker ( F _{i} K ^{r} \to \gr ^F _i K ^{r+1})$
induces the map 
$\phi 
\colon 
\ker d ^r  \cap F _i K ^r \to \ker ( F _{i} K ^{r} \to \gr ^F _i K ^{r+1})
/
d _i ^{r-1} (F _i K ^{r-1}) + F _{i-1} K ^{r} $. 
Let
$x \in \ker d ^r  \cap F _i K ^r$ be an element in the kernel of $\phi$. 
Then there exist
$y \in F _i K ^{r-1}$ and $z \in F _{i-1} K ^{r} $ such that
$x = d _i ^{r-1} (y) +z$. Since $d _i ^{r-1} (y)\in \ker d ^r$, 
we get $z \in \ker d ^r$ and then $z \in\ker d ^r  \cap F _{i-1} K ^r $. 
Hence $\ker \phi \subset d _i ^{r-1} (F _i K ^{r-1}) + \ker d ^r  \cap F _{i-1} K ^r $.
Since the converse is obvious, we get
$\ker \phi = d _i ^{r-1} (F _i K ^{r-1}) + \ker d ^r  \cap F _{i-1} K ^r $
From \ref{lem-subquotientHgrdefpre2}, 
this yields that
$L$ is a subobject of 
$H ^{r} (\gr ^F  _{i} K ^{\bullet})$.
Moreover, 
we have the epimorphism 
$\ker d ^r  \cap F _i K ^r \to 
\ker d ^r  \cap F _i K ^r + \im d ^{r-1} /\ker d ^r  \cap F _{i-1} K ^r + \im d ^{r-1}$.
Since 
$d _i ^{r-1} (F _i K ^{r-1}) + \ker d ^r  \cap F _{i-1} K ^r $ 
is in the kernel of this map, 
we get the factorisation
$L \to 
\ker d ^r  \cap F _i K ^r + \im d ^{r-1} /\ker d ^r  \cap F _{i-1} K ^r + \im d ^{r-1}$,
which is still an epimorphism.
From 
\ref{lem-subquotientHgrdefpre1},
this implies that
$L$ is a quotient of $\gr ^F  _{i} (H ^{r} (K ^{\bullet}))$.
Hence, 
$\gr ^F  _{i} (H ^{r} (K ^{\bullet}))$ is a quotient of a submodule of 
$H ^{r} (\gr ^F  _{i} K ^{\bullet})$.

\end{proof}

\begin{empt}
Let $(M,M _i)$ and $(N, N _i)$ be two filtered $(D, D  _i)$-modules.
For any integer $i \in \Z$, 
let $F _i \mathrm{Hom} _{D} ( M , N) $ be the subgroup  of $Hom _{D} ( M , N)$ of  the elements $\phi$ such that, 
for any integer $j\in \Z$, $\phi ( M _{j}) \subset N _{i +j}$.
For any integer $j \in \Z$, for any $\phi \in F _i \mathrm{Hom} _{D} ( M , N)$, 
we get a morphism
$\gr _i M \to \gr _{i+j} N$
defined by sending the class of a element 
$x \in M _i$ to the class of $\phi ( x)$.
Hence $\phi$ induces a map
$\gr M \to \gr N$, which is in fact $\gr _D $-linear. 
Hence we get the canonical  morphism
$F _i \mathrm{Hom} _{D} ( M , N)
\to 
\mathrm{Hom} _{\gr D}
 (\gr  M, \gr N)$
 and then 
\begin{equation}
\label{grHom2Homgr}
 \gr ^F  \mathrm{Hom} _{D} ( M , N)
\to 
\mathrm{Hom} _{\gr D}
 (\gr  M, \gr N).
\end{equation}
\end{empt}

\begin{lem}
\label{grHom2Homgr-free}
Let $(L,L _i)$ be a free filtered $(D, D  _i)$-module of finite type.
Let $(N, N _i)$ be  filtered $(D, D  _i)$-module.
The canonical morphism 
$$ \gr  \mathrm{Hom} _{D} ( L , N)
\to 
\mathrm{Hom} _{\gr D}
 (\gr  L, \gr N)
$$
of \ref{grHom2Homgr} is an isomorphism of abelian groups.
\end{lem}

\begin{proof}
Let $(M,M _i)$, $(M',M '_i)$ be two filtered $(D, D  _i)$-modules
and $(M'',M ''_i):= (M,M _i) \oplus (M',M '_i)$. 
Then the morphism of \ref{grHom2Homgr}
$ \gr ^F  \mathrm{Hom} _{D} ( M'' , N)
\to 
\mathrm{Hom} _{\gr D}
 (\gr  M'', \gr N)$ is an isomorphism if and only if 
so are 
$ \gr ^F  \mathrm{Hom} _{D} ( M , N)
\to 
\mathrm{Hom} _{\gr D}
 (\gr  M, \gr N)$
 and 
 $ \gr ^F  \mathrm{Hom} _{D} ( M' , N)
\to 
\mathrm{Hom} _{\gr D}
 (\gr  M', \gr N)$. 
 Hence, we can suppose
 that 
 $(L, L _i) = ( D ( n ) , D ( n) _i)$, for some integer $n$. 
Twisting the filtrations, we can suppose $n =0$. 
Finally, we compute that 
the morphism of 
\ref{grHom2Homgr}
$\gr  \mathrm{Hom} _{D} ( D , N)
\to 
\mathrm{Hom} _{\gr D}
 (\gr  D, \gr N)$
 is, modulo the identifications 
 $N = \mathrm{Hom} _{D} ( D , N)$
 and 
$ \mathrm{Hom} _{\gr D}
 (\gr  D, \gr N)
 =
 \gr N$,
 the identity, which is an isomorphism.
\end{proof}

\begin{prop}
\label{Extgr-nogr}
Let $(M, M _i)$ be a good filtered $(D, D  _i)$-module.
Let $(N, N _i)$ be a filtered $(D, D  _i)$-module.
For any integer $r$, there exists a canonical filtration $F$ of 
$\Ext ^r _{D} (M , N ) $ 
satisfying the following properties
\begin{enumerate}
\item $\Ext ^r _{D} (M , N )  = \cup _{i\in \Z} F _{i} \,\Ext ^r _{D} (M , N ) $,
\item $\gr ^{F} \Ext ^r _{D} (M , N ) $ is a subquotient of 
$\Ext ^r _{\gr D } (\gr M , \gr N)$. 

\item Suppose $(N, N _i)=(D, D  _i)$. 
The canonical filtration $(F _{i}\, \Ext ^r _{D} (M , D)) _{i\in \Z} $
of the right $D$-module $\Ext ^r _{D} (M , D) $
is a good filtration.
In particular, $0  = \cap _{i\in \Z} F _{i}\, \Ext ^r _{D} (M , D)$.
Moreover, we have the implication 
$$\Ext ^r _{\gr D } (\gr M , \gr D)=0
\Rightarrow
\Ext ^r _{D} (M , D) =0.$$
\end{enumerate}
\end{prop}

\begin{proof}
From \ref{4.2.3.2Malgrange} and with its definition,
there exists a good resolution $(L _{\bullet}, L  _{\bullet, i}) $
of $(M, M _i)$.
We put 
$K ^{\bullet}:= \mathrm{Hom} _{D} ( L _{\bullet}, N)$.
Since $L _{\bullet}$ is a resolution of 
$M$ by projective $D$-modules, we get 
$H ^{r} (K ^{\bullet})= 
\Ext ^r _{D} (M , N )$.

Let $F _{i} K ^{n}$ be the subset of  
the elements $\phi$ of $Hom _{D} ( L _{n}, N)$ such that, 
for any integer $j\in \Z$, $\phi ( L _{n, j}) \subset N _{i +j}$.
Since 
$L _n$ is 	a $D$-module of finite type, we get
$\cup _{i \in \Z} F _{i} K ^{n} = K ^{n}$.
With the canonical induced filtration on 
 $H ^{r} (K ^{\bullet})= 
\Ext ^r _{D} (M , N )$
 (see \ref{lem-subquotientHgr}), 
this yields the first property.
 Since 
$L _n$ is a free 
filtered $(D, D  _i)$-modules of finite type, from Lemma \ref{grHom2Homgr-free}, 
the canonical morphism
$\gr    K ^{n}
\to 
\mathrm{Hom} _{\gr D}
 (\gr  L _{n}, \gr N)$
 is an isomorphism.
Since $\gr  L _{\bullet}$ is a resolution of 
$\gr M$ by projective $\gr D $-modules, we get
 $H ^{r} (\gr   K ^{\bullet})
 =\Ext ^r _{\gr D } (\gr M , \gr N)$.
This implies the second point by using Lemma \ref{lem-subquotientHgr}.

When $(N, N _i)=(D, D  _i)$, the filtration
$F _{i} K ^{n}$ of the right $D$-module $K ^{n}$
is a good filtration. We denote by $d _n \colon K ^{n} \to K ^{n+1}$ the canonical morphisms. 
From \cite[A.1.1.2]{Laumon-TransfCanon}, 
the induced filtrations on $\ker d _n$
and next on $\ker d _n / Im \,d _{n-1}$ (induced from the surjection
$\ker d _n \to \ker d _n / Im \,d _{n-1}$)
are good. 
We notice that this filtration on $\ker d _n / \mathrm{Im} \,d _{n-1}=H ^{n} (K ^{\bullet})$
is the same as that defined at \ref{lem-subquotientHgrdef}, which is the first assertion of the third point.
With the second point, this yields the rest of the third point. 
\end{proof}

\section{Purity of the characteristic variety of holonomic $F$-$\D$-modules}

\begin{lemm}
\label{lemm-codimmodp}
Let $X$ be an affine smooth variety over $k$, 
$\overline{D}: =\Gamma (X, \D ^{(0)} _{X/k})$,
$(\overline{D} _i) _{i\in \N}$ be its filtration by the order of the operators, 
$\overline{f}$ be an homogeneous element of 
$\gr\, \overline{D}$.
Let $(\overline{M}, \overline{M} _i)$ and $(\overline{N}, \overline{N} _i)$ be two good filtered $(\overline{D} _{[\overline{f}]}, \overline{D}  _{[\overline{f}],i})$-modules
and $r$ be an integer. 
\begin{enumerate}
\item We have $\mathrm{codim} \, \Ext ^r _{\gr \overline{D} _{[\overline{f}]}} (\gr \overline{M} , \gr \overline{N}) \geq r$.

\item If $r < \mathrm{codim} \, \gr \overline{M}$ then $\Ext ^r _{\gr \overline{D} _{[\overline{f}]}} (\gr \overline{N} , \gr \overline{D} _{[\overline{f}]}) =0$.
\end{enumerate}

\end{lemm}

\begin{proof}
By construction (see \cite[A.2]{Laumon-TransfCanon}), 
we get 
$\gr \overline{D} _{[\overline{f}]}\riso (\gr \, \overline{D} )_{\overline{f}}$.
Then, this is well known (e.g. see \cite[D.4.4]{HTT-DmodPervSheavRep}).
\end{proof}

\begin{empt}
[Localisation and $\pi$-adic completion]
\label{loca-completion}
Let $\X$ be an affine smooth $\V$-formal scheme, $X _n$ be the reduction of $\X$ modulo $\pi ^{n+1}$. 
We put $D: =\Gamma (X, \D ^{(0)} _{\X/\S})$
and  $D _n: =\Gamma (X, \D ^{(0)} _{X _n/S _n})$. These rings are canonically filtered by the order of the differential operators ; 
we denote by  
$(D ,  D _{i} )$ and $(D _{n}, D _{n , i})$ the 
(ind-pro) complete filtered rings. 
Let  $f$ be an homogeneous element of 
$\gr\, D$ and $f _n$ be its image in 
$\gr\, D _n$.
With the notation of \ref{nota-local}, 
by using the same arguments as in the proof of \cite[2.3]{Abe-microdiff}, we get the canonical isomorphism of (ind-pro) complete filtered ring

\begin{equation}
\label{microdiffchgbase}
(D _{[f]} ,  D _{[f], i} ) \otimes _{\V}  \V / \pi ^{n+1} 
\riso 
(D _{n [f _n]}, D _{n [f _n], i}).
\end{equation}

We put $\widehat{D} _{[f]}$ (be careful that this notation is slightly different from that of 
\ref{nota-local})
as the $\pi$-adic completion of 
$D _{[f]}$, i.e. 
$\widehat{D} _{[f]}:= \underset{n}{\underleftarrow{\lim}}\,D _{[f]} / \pi ^{n+1} D _{[f]}
\riso  \underset{n}{\underleftarrow{\lim}} D _{n, [f _n]}$.
Using Corollary \cite[A.1.1.1]{Laumon-TransfCanon} and \cite[3.2.2.(iii)]{Be1}, 
we get from the isomorphism 
\ref{microdiffchgbase} the noetherianity of 
$\widehat{D} _{[f]}$.

Finally, when there is no confusion 
with the notation \ref{nota-local}, 
for any coherent $\widehat{D}$-module (resp. coherent $\widehat{D} _\Q$-module) $M$ (resp. $N$), 
we set (by default in this new context)
$M _{[f]} := \widehat{D} _{[f]} \otimes _{\widehat{D}} M$
(resp. $N _{[f]} := \widehat{D} _{[f]} \otimes _{\widehat{D}} N$).

\end{empt}

\begin{lemm}
\label{flat-local-compl}
With the notation of \ref{loca-completion}, the homomorphism 
$\widehat{D} \to \widehat{D} _{[f]}$
is flat.
\end{lemm}

\begin{proof}
This is a consequence of \cite[3.2.3.(vii)]{Be1},
\cite[A.2.3.4.(ii)]{Laumon-TransfCanon}
and \ref{microdiffchgbase}.
\end{proof}

\begin{rem}
\label{rem-loca-completion}
With the notation of \ref{loca-completion}, 
let $M$ be a coherent $\widehat{D}$-module.
We put 
$M _n := M /  \pi ^{n+1} M$. From \cite[5.2.3.(iv)]{Beintro2}, there exists a good filtration 
$(M _{n,i}) _{i\in \N}$ of $M _n$ indexed by $\N$. 
We recall (see notation \ref{nota-local-defM}) that we get a good filtered 
$(D _{n, [f _n]}, D _{n, [f _n],i})$-module by putting 
$(M _{n, [f _n]}, M _{n, [f _n],i}): = (D _{n, [f _n]}, D _{n, [f _n],i}) \otimes _{ (D _n, D _{n, i}) } (M _n, M _{n, i})$.
Moreover, since 
$\widehat{D} _{[f]}
/  \pi ^{n+1} 
\widehat{D} _{[f]}
\riso 
D _{n, [f _n]}$ (use \ref{microdiffchgbase}), 
then 
\begin{equation}
\label{323Be1}
M _{[f]} /  \pi ^{n+1} M _{[f]} 
\riso M _{n, [f _n]}.
\end{equation}

From \cite[3.2.3.(v)]{Be1}, since $M _{[f]}$ is a $\widehat{D} _{[f]}$-module of finite type, 
then $M _{[f]}$ is complete for the $\pi$-adic topology.
 Hence, using \ref{323Be1}
 we get the canonical isomorphism of $\widehat{D} _{[f]}$-modules
$M _{[f]}\riso  \underset{n}{\underleftarrow{\lim}} M _{n, [f _n]}$.

\end{rem}

\begin{empt}
[Characteristic variety (of level $m$)]
Let $\X$ be a smooth $\V$-formal scheme, $X $ be the reduction of $\X$ modulo $\pi $.
Let $m\in \N$ be an integer. 
Let us recall Berthelot's definition of characteristic varieties (of level $m$) as explained in \cite[5.2]{Beintro2}. 
\begin{enumerate}
\item Let $\G$ be a coherent 
$\D _{X} ^{(m)}$-module. 
Berhelot defined the cotangent space of level $m$ of $X$ 
defined by putting 
$T ^{(m)*} X := (\Spec \gr \D _{X} ^{(m)} ) _{\mathrm{red}}$, 
where $\D _{X} ^{(m)}$ is filtered by the order.
Choose a good filtration $(\G _n) _{n\in \N}$ (see the definition \cite[5.2.3]{Beintro2}),
i.e. a filtration such that $\gr \G$ is a $\gr \D _{X} ^{(m)}$, 
where $\D _{X} ^{(m)}$ is filtered by the order. 
The characteristic variety of level $m$ of $\G$, denoted by 
$\mathrm{Car} ^{(m)} (\G)$ is by definition the support of 
$\gr \G$ in $T ^{(m)*} X $.
Berthelot checked that this is well defined. 
In particular, we see that $\G =0$ if and only if 
$\mathrm{Car} ^{(m)} (\G)$ is empty. 

\item Let  $\FF$ be a coherent 
$\widehat{\D} _{\X} ^{(m)}$-module. 
The characteristic variety of level $m$ of $\FF$ 
$\mathrm{Car} ^{(m)} (\FF)$ is by definition the 
characteristic variety of level $m$ of $\FF/\pi \FF$
as coherent 
$\D _{X} ^{(m)}$-module, i.e. 
$\mathrm{Car} ^{(m)} (\FF):= \mathrm{Car} ^{(m)} (\FF/\pi \FF)$.

\item Let $\E$ be a coherent 
$\widehat{\D} _{\X,\Q} ^{(m)}$-module. 
Choose a  coherent 
$\widehat{\D} _{\X} ^{(m)}$-module $\overset{\circ}{\E}$ without $p$-torsion such
that there exists an isomorphism of $\widehat{\D} _{\X,\Q} ^{(m)}$-modules of the form
$\overset{\circ}{\E} _\Q \riso \E$.
The characteristic variety of level $m$ of $\E$ denoted by 
$\mathrm{Car} ^{(m)} (\E)$ is by definition that 
of $\overset{\circ}{\E}$ as coherent $\widehat{\D} _{\X} ^{(m)}$-module, i.e., 
 $\mathrm{Car} ^{(m)} (\E):= \mathrm{Car} ^{(m)} (\overset{\circ}{\E}/\pi \overset{\circ}{\E})$.
 Berthelot checked that this is well defined. 
 Moreover, $\mathrm{Car} ^{(m)} (\E)$ is empty if and only if $\E =0$
 (because, since $\overset{\circ}{\E}$ has no $p$-torsion,  $\overset{\circ}{\E}/\pi \overset{\circ}{\E} =0$ is equivalent to 
 $\overset{\circ}{\E} _\Q=0$).

\item Let $(\NN,\phi)$ be a coherent $F\text{-}\D ^{\dag} _{\X,\Q}$-module, 
i.e. a coherent $\D ^{\dag} _{\X,\Q}$-module $\NN$ and an isomorphism of $\D ^{\dag} _{\X,\Q}$-modules $\phi$  of the form
$\phi \colon F ^* \NN \riso \NN$.
Then there exists a (unique up to canonical isomorphism) coherent $\widehat{\D} _{\X,\Q} ^{(0)}$-module
$\NN ^{(0)}$ and an isomorphism 
$\phi ^{(0)}\colon \widehat{\D} _{\X,\Q} ^{(1)} \otimes _{\widehat{\D} _{\X,\Q} ^{(0)}}\NN ^{(0)} \riso F ^* \NN ^{(0)}$
which induces canonically $\phi$.
Then, by definition, 
the characteristic variety of $\NN$ denoted by 
$\mathrm{Car} (\NN)$
is by definition
the  characteristic variety of level $0$ of $\NN ^{(0)}$,
i.e., 
$\mathrm{Car} (\NN ):=  \mathrm{Car} ^{(0)} (\NN ^{(0)})$.
We have 
$\mathrm{Car}  (\NN)$ is empty if and only if $\NN =0$
\end{enumerate}
\end{empt}

\begin{lemm}
\label{lemmN_f=0}
We keep notation \ref{loca-completion}.
Let $\NN$ be a coherent $\widehat{\D} ^{(0)} _{\X,\Q}$-module,
$\mathrm{Car} ^{(0)} (\NN)$ its characteristic variety of level $0$
(see the definition in \cite[5.2.5]{Beintro2}). We put $N: =\Gamma (X, \NN)$. 
The following assertions are equivalent
\begin{enumerate}
\item $D (f _{0}) \cap \mathrm{Car} ^{(0)} (\NN) = \emptyset$. 
\item $N _{[f]}=0$.
\end{enumerate}

\end{lemm}

\begin{proof}
From  \cite[3.4.5]{Be1}, there exits a coherent $\widehat{D}$-module without $p$-torsion $M$
such that $M _{\Q} \riso N$.
Since the extension
$\widehat{D} \to \widehat{D} _{[f]}$
is flat (see \ref{flat-local-compl}), 
we get that $M _{[f]}$ is also without $p$-torsion 
($p$ is in the center of $\widehat{D}$ and $\widehat{D} _{[f]}$).
This yields that 
$N _{[f]}=0$
if and only if 
$M _{[f]}=0$.
Let $\overline{M} := M / \pi M $. 
From \ref{323Be1}, we have
$M _{[f]}
/ \pi M _{[f]}
\riso 
\overline{M} _{[f _0]}$.
Hence, 
$M _{[f]}=0$ if and only if 
$\overline{M} _{[f _0]}=0$ (e.g. see \cite[3.2.2.(ii)]{Be1}).
From \cite[5.2.3.(iv)]{Beintro2}, there exists a good filtration 
$(\overline{M} _{i}) _{i\in \N}$ of $\overline{M}$ indexed by $\N$.
From the remark \ref{rem-loca-completion}, 
this induces canonically the (ind-pro) complete  
$(\overline{D} _{[f _0]}, \overline{D} _{[f _0],i})$-module
$(\overline{M} _{[f _0]}, \overline{M} _{[f _0],i})$.
Since $\overline{M} _{[f _0]}$ is (ind-pro) complete, then
the equalities $\overline{M} _{[f _0]}=0$ and 
$ \gr \, (\overline{M} _{[f _0]}, \overline{M} _{[f _0],i})) =0$ are equivalent. 
Also,
$(\gr \, \overline{M} ) _{f _0} =0$
if and only if 
$D (f _{0}) \cap \mathrm{Supp} (\gr \, \overline{M}) = \emptyset$.
Since $(\gr \, \overline{M} ) _{f _0} \riso \gr \, (\overline{M} _{[f _0]})$ (see \cite[A.1.1.3]{Laumon-TransfCanon})
and since by definition
$\mathrm{Car} ^{(0)} (\NN) 
=
\mathrm{Supp} (\gr \, (\overline{M}, \overline{M} _i))$,
we conclude the proof.
\end{proof}

\begin{rem}
\label{rem-homogenous}
Let $A = \oplus _{i \in \N} A _{i}$ be a graded ring. Let $I$ be a graded ideal. 
Let $a _{1}, \dots, a _{r}$ be some homogeneous generators of $I$.
We notice that $|\Spec A |\setminus V (I)= \cup _{i=1} ^{r} D ( a _i)$.
\end{rem}

The following proposition is the analogue of 
\cite[2.11]{Kashiwara-B-function-Hol}:
\begin{prop}
\label{prop-extr=0}
Let $\X$ be a smooth $\V$-formal scheme.
Let $\NN$ be a coherent $\widehat{\D} ^{(0)} _{\X,\Q}$-module and 
$V$ be an irreducible component of codimension $r$ of $\mathrm{Car} ^{(0)} (\NN)$, 
the characteristic variety of level $0$
of $\NN$
(see \cite[5.2.5]{Beintro2}).
Then, 
$\mathrm{Car} ^{(0)} (\mathcal{E} xt ^{r} _{ \widehat{\D} ^{(0)} _{\X,\Q}} (\NN, \widehat{\D} ^{(0)} _{\X,\Q}))$ contains $V$.
\end{prop}

\begin{proof}
We follow the proof of \cite[2.11]{Kashiwara-B-function-Hol}: 
first, we can suppose $\X$ affine with local coordinates. 
We set 
$D: =\Gamma (\X, \D ^{(0)} _{\X/\S})$,
$N: =\Gamma (X, \NN)$,
$\overline{D}: =\Gamma (X, \D ^{(0)} _{X/S})$.
Let $M$ be a coherent $\widehat{D}$-module without $p$-torsion
such that $M _{\Q} \riso N$.
Let $\overline{M} := M / \pi M $. 
From \cite[5.2.3.(iv)]{Beintro2}, there exists a good filtration 
$(\overline{M} _{i}) _{i\in \N}$ of $\overline{M}$ indexed by $\N$. 
By definition, we have
$\mathrm{Car} ^{(0)} (\NN) 
=
\mathrm{Supp} (\gr \, (\overline{M}, \overline{M} _i))$ 
(we recall that this is independent on the choice of the good filtration).
Let $\eta$ be the generic point of $V$. 
From \cite[7.D and 10.B.i)]{Matsumura-Comm-Alg}, 
the irreducible components of 
$\mathrm{Supp} (\gr \, \overline{M})$ are of the form 
$V(J)$ with $J$ a homogeneous ideal. 
Let $Z$ be the union of the irreducible components of 
$\mathrm{Supp} (\gr \, \overline{M})$ which do not contain $\eta$.
Then, we get from the remark \ref{rem-homogenous} that 
there exists a homegeneous element $f \in \gr \, D$ such that
$\eta \in D ( f _0)$  and $D ( f _0) \cap Z = \emptyset$ 
(in other words, $D ( f _0) \cap \mathrm{Car} ^{(0)} (\NN)  = D ( f _0) \cap V \not = \emptyset$).

Now, suppose absurdly that 
$\eta \not \in \mathrm{Car} ^{(0)}(\mathcal{E} xt ^{r} _{ \widehat{\D} ^{(0)} _{\X,\Q}} (\NN, \widehat{\D} ^{(0)} _{\X,\Q}))$. 
Using the same arguments as above, 
there exists a homogeneous element $g \in \gr \, D$ such that
$\eta \in D ( g _0)$  and $D ( g _0) \cap \mathrm{Car} ^{(0)} (\mathcal{E} xt ^{r} _{ \widehat{\D} ^{(0)} _{\X,\Q}} (\NN, \widehat{\D} ^{(0)} _{\X,\Q}))  
=\emptyset$. We put $h= fg$. 
Hence, we have 
$\eta \in D ( h _0)$ and 
$D ( h _0) \cap \mathrm{Car} ^{(0)} (\NN)  = D ( h _0) \cap V$ and
$D ( h _0) \cap \mathrm{Car} ^{(0)} (\mathcal{E} xt ^{r} _{ \widehat{\D} ^{(0)} _{\X,\Q}} (\NN, \widehat{\D} ^{(0)} _{\X,\Q}))  
=\emptyset$.

1) Since $\mathcal{E} xt ^{r} _{ \widehat{\D} ^{(0)} _{\X,\Q}} (\NN, \widehat{\D} ^{(0)} _{\X,\Q})$ is a coherent (right) 
$\widehat{\D} ^{(0)} _{\X,\Q}$-module, from Theorem A and B of Berthelot (see \cite[3]{Be1}), we get the equality
$\Gamma (\X, \mathcal{E} xt ^{r} _{ \widehat{\D} ^{(0)} _{\X,\Q}} (\NN, \widehat{\D} ^{(0)} _{\X,\Q}))
=\mathrm{Ext} ^{r} _{ \widehat{D}  _{\Q}} ( N, \widehat{D}  _{\Q})$.
From \ref{lemmN_f=0}, this implies
$\mathrm{Ext} ^{r} _{ \widehat{D}  _{\Q}} ( N, \widehat{D}  _{\Q})
\otimes _{ \widehat{D}  _{\Q} }  
 \widehat{D}  _{[h],\Q}
 =0$.
Since the extension 
$\widehat{D}  _{\Q}
\to 
\widehat{D}  _{[h],\Q}$ is flat (see \ref{flat-local-compl}),
we get
$ \mathrm{Ext} ^{r} _{ \widehat{D}  _{[h],\Q}} ( N _{[h]},  \widehat{D}  _{[h],\Q}) 
\riso
\mathrm{Ext} ^{r} _{ \widehat{D}  _{\Q}} ( N, \widehat{D}  _{\Q})
\otimes _{ \widehat{D}  _{\Q} }  
 \widehat{D}  _{[h],\Q}
 =0$.

2) a) Since 
$\mathrm{Car} ^{(0)} (\NN) 
=
\mathrm{Supp} (\gr \, \overline{M})$,
then 
$D (h _0) \cap \mathrm{Car} ^{(0)} (\NN) 
=
\mathrm{Supp}  (\gr \, \overline{M} ) _{h _0})$.
Since we have also
$D ( h _0) \cap \mathrm{Car} ^{(0)} (\NN)  = D ( h _0) \cap V$, 
then in particular we get
$\mathrm{Codim} (\gr \, \overline{M} ) _{h _0} =r$.
Since $ (\gr \, \overline{M} ) _{h _0}
=\gr \, (\overline{M}  _{[h _0]} )$, 
then from \ref{lemm-codimmodp} for any $i<r$ we obtain 
$\Ext ^i _{\gr \overline{D} _{[h _0]}} (\gr (\overline{M} _{[h _0]}), \gr \overline{D} _{[h _0]}) =0$.
From \ref{Extgr-nogr}.3, this yields that
for $i < r$, 
$\Ext ^i _{ \overline{D} _{[h _0]}} ( \overline{M} _{[h _0]} ,  \overline{D} _{[h _0]}) =0$.
On the other hand, from \ref{lemm-codimmodp} we get for any $i >r$ the inequality
$\mathrm{Codim} 
(\Ext ^i _{\gr \overline{D} _{[h _0]}} (\gr (\overline{M} _{[h _0]}), \gr \overline{D} _{[h _0]}))
>r$.
Hence,  
by reducing $D (h _0)$ if necessary (use again the remark \ref{rem-homogenous}),
for any $i>r$ we get 
$\Ext ^i _{\gr \overline{D} _{[h _0]}} (\gr (\overline{M} _{[h _0]}), \gr \overline{D} _{[h _0]})=0$ 
and then 
$\Ext ^i _{ \overline{D} _{[h _0]}} ( \overline{M} _{[h _0]} ,  \overline{D} _{[h _0]}) =0$.
To sum up, 
we have found
an homogeneous element $h \in \gr \, D$ such that
$\eta \in D (h _0)$ and for $i\not =r$, 
$\Ext ^i _{ \overline{D} _{[h _0]}} ( \overline{M} _{[h _0]} ,  \overline{D} _{[h _0]}) =0$.

2) b)  Now, since $M _{[h]}$ is without $p$-torsion, 
$\R \mathrm{Hom} _{ \widehat{D}  _{[h]}} ( M _{[h]},  \widehat{D}  _{[h]}) \otimes ^{\L} _{\widehat{D}  _{[h]}}
\overline{D} _{[h _0]}
\riso 
\R \mathrm{Hom} _{ \overline{D} _{[h _0]}} ( \overline{M} _{[h _0]} ,  \overline{D} _{[h _0]})$.
From the exact sequence of universal coefficients 
(e.g. see the beginning of the proof of \cite[I.5.8]{virrion}), we get the inclusion
$\mathrm{Ext} ^{i} _{ \widehat{D}  _{[h]}} ( M _{[h]},  \widehat{D}  _{[h]}) \otimes _{\widehat{D}  _{[h]}}
\overline{D} _{[h _0]} \hookrightarrow \Ext ^i _{ \overline{D} _{[h _0]}} ( \overline{M} _{[h _0]} ,  \overline{D} _{[h _0]})$.
Hence, for any $i \not =r$, from the step 2) a) of the proof,
we obtain the vanishing
$\mathrm{Ext} ^{i} _{ \widehat{D}  _{[h]}} ( M _{[h]},  \widehat{D}  _{[h]}) \otimes _{\widehat{D}  _{[h]}}
\overline{D} _{[h _0]} =0$.
By using \cite[3.2.2.(ii)]{Be1}, since 
$\mathrm{Ext} ^{i} _{ \widehat{D}  _{[h]}} ( M _{[h]},  \widehat{D}  _{[h]})$
is a coherent 
$\widehat{D}  _{[h]}$-module, for 
$i \not = r$ we get 
$\mathrm{Ext} ^{i} _{ \widehat{D}  _{[h]}} ( M _{[h]},  \widehat{D}  _{[h]})=0$ and
then 
$\mathrm{Ext} ^{i} _{ \widehat{D}  _{[h],\Q}} ( N _{[h]},  \widehat{D}  _{[h], \Q})=0$
(because 
$\widehat{D}  _{[h]} \to \widehat{D}  _{[h], \Q}$ is flat). 

3) From steps 1) and 2), 
we have checked that
$\R \mathrm{Hom} _{ \widehat{D}  _{[h],\Q}} ( N _{[h]},  \widehat{D}  _{[h], \Q})=0$.
By using the biduality isomorphism (see \cite[I.3.6]{virrion} and notice that $N _{[h]}$ is a perfect complex because so is 
$N$ and because the extension $\widehat{D}  _{\Q}
\to 
\widehat{D}  _{[h],\Q}$ is flat), we get 
$N _{[h]}=0$, which is absurd following Lemma \ref{lemmN_f=0} because $\eta \in D (h _0)$.
\end{proof}

\begin{thm}
\label{puritythm}
Let $\X$ be a smooth $\V$-formal scheme.
Let $r$ be an integer, $\NN$ be a coherent $\widehat{\D} ^{(0)} _{\X,\Q}$-module such that
$\mathcal{E} xt ^{s} _{ \widehat{\D} ^{(0)} _{\X,\Q}} (\NN, \widehat{\D} ^{(0)} _{\X,\Q})=0$
for any $s \not = r$. 
Then, the characteristic variety 
$\mathrm{Car} ^{(0)} (\NN)$ 
of $\NN$ is purely of codimension $r$.

\end{thm}

\begin{proof}
If $V$ is an irreducible component of 
$\mathrm{Car} ^{(0)} (\NN)$ of codimension $s$, then 
from \ref{prop-extr=0} we get
$\mathcal{E} xt ^{s} _{ \widehat{\D} ^{(0)} _{\X,\Q}} (\NN, \widehat{\D} ^{(0)} _{\X,\Q})\not =0$
since it contains $V$. Hence $s =r$.
\end{proof}

\begin{coro}
\label{coro-puritythm}
Let $\X$ be a smooth integral $\V$-formal scheme of dimension $d$.
Let $\NN\not =0$ be a holonomic $F\text{-}\D ^{\dag} _{\X,\Q}$-module.
Then, the characteristic variety 
$\mathrm{Car}  (\NN)$ 
of $\NN$ is purely of codimension $d$.

\end{coro}

\begin{proof}
The is a consequence of Virrion's holonomicity characterization (see Theorem 
\cite[III.4.2]{virrion} and of Theorem \ref{puritythm}.
\end{proof}

\section{Lagrangianity for log-extendable overconvergent isocrystal}

\begin{ntn}
\label{stab-T_X X}
Let $X$ be a smooth $k$-variety. 
For any an quasi-coherent $\O _X$-module $\E$,
we denote by $\sym (\E) $ the symetric algebra of $\E$  
and by $\mathbb{V} (\E): = \Spec (\sym (\E) )$ endowed with its canonical projection
$\mathbb{V} (\E) \to \Spec \sym (\O _X) = X$.
We denote by $\Omega ^{1} _X $ the sheaf of differential form of $X/\Spec (k)$ (we skip $k$ in the notation),
and $\mathcal{T} _X$ the tangent space of $X/\Spec (k)$, i.e. the $\O _X$-dual of $\Omega ^{1} _X $.
We denote by $T ^{*} X:= \mathbb{V} (\mathcal{T} _X)$ the cotangent space of $X$
and $\pi _X \colon T ^{*} X \to X$ the canonical projection.
Recall that from \cite[1.7.9]{EGAII}, there is a canonical bijection between
sections of $\pi _X$ and $\Gamma (X, \Omega ^{1} _X)$.
We denote by $T ^* _X X$ the section corresponding to the zero section of 
$\Gamma (X, \Omega ^{1} _X)$.
If $t _1, \dots, t _d$ are local coordinates of $X$, we get local coordinates
$t _1 , \dots, t _d, \xi _1, \dots, \xi _d$ of $T ^{*} X$, where $\xi _i$ is the element associated with
$\partial _i$, the derivation with respect to $t _i$. 
Is this case, $T ^* _X X= V ( \xi _1, \dots, \xi _d)$ is the closed subvariety of 
$T ^* X$ defined by $\xi _1 =0, \dots, \xi _d =0$.

Let $f \colon X \to Y$ be a morphism of smooth $k$-varieties. 
Using the equality \cite[1.7.11.(iv)]{EGAII}
we get the last one
$X \times _{Y} T ^* Y =
X \times _{Y} \mathbb{V} (\mathcal{T}_Y) 
=\mathbb{V} (f ^*\mathcal{T} _Y)$.
The morphism $f ^* \Omega ^{1} _Y \to \Omega ^{1} _X $ induced by $f$
yields by duality 
$\mathcal{T} _X \to f ^* \mathcal{T} _Y$ and then by functoriality
$\mathbb{V} (f ^*\mathcal{T} _X)
\to \mathbb{V} (\mathcal{T} _Y) = T ^* Y$.
By composition, we get the morphism
denoted by
$\rho _{f}\colon  X \times _{Y} T ^* Y \to T ^{*} X $ this morphism (induced by $f$).
We define the $k$-variety $T ^{*} _{X} Y$ (recall a $k$-variety is a separated reduced scheme of finite type over $k$
from our convention) 
by setting $T ^{*} _{X} Y := \rho _{f} ^{-1} ( T ^* _X X)$.

Denote by $\alpha _X$ the composition of the diagonal morphism 
$\Delta _{T ^* X /X}
\colon
T ^* X
\hookrightarrow 
T ^* X \times _{X}T ^* X$ with 
$\rho _{T _X} \colon T ^* X \times _{X}T ^* X \to T ^* (T ^* X)$ 
(replace $f$ by $\pi _X$ in the definition above).
We check that $\alpha _X$ is a section of the canonical morphism
$\pi _{T ^* X}\colon 
T ^* (T ^* X)
\to 
T ^* X$ (indeed, since this is local in $T ^* X$, we check it  using local coordinates).
Hence, $\alpha _X$ correspond to 
a section $\Gamma ( X ,\Omega ^1 _{T ^* X})$, which we still denote by 
$\alpha _X$.
Similarly as in the very beginning of \cite[E.2]{HTT-DmodPervSheavRep}, we say
that $\alpha _X$ is the canonical $1$-form  of $T ^{*} X$.
In this local coordinate system, we get
$\alpha _X = \sum _{i =1} ^{d} \xi _i d t _i $.
\end{ntn}

\begin{dfn}
\label{Def-dev-Lagrangianity}
Let $\X$ be a smooth formal scheme over $\V$ and $X$ be its special fiber.

\begin{enumerate}
\item Let $E$ be a subvariety of $T ^{*} X$. The restriction $\alpha _X $ on $E$ is denoted by $\alpha _X |E \in \Omega ^{1} _E$. 
As before Proposition \cite[2.3]{Kashiwara-B-function-Hol}, 
we say that $E$ is isotropic if there exists an open dense subset $U$ of $E$ such that 
$\alpha _X |U =0$.
When $X$ is purely of dimension $d$, we say that $E$ is Lagrangian if $E$ 
is isotropic and purely of codimension $d$.
In general, we say that $E$ is Lagrangian is the restriction of $E$ on the irreducible components of $T ^* X$ are Lagrangian in the previous sense.

\item Let $\E$ be a coherent $F\text{-}\D ^{\dag} _{\X,\Q}$-module. We say that 
$\E$ is Lagrangian (resp. isotropic)
if $\mathrm{Car} (\E ) $ is Lagrangian (resp. isotropic).
From \ref{coro-puritythm}, $\E$ is Lagrangian if and only if $\E$ is holonomic and isotropic. 

\item Let $\E$ be a complex of  $F \text{-}D ^{\mathrm{b}} _{\mathrm{coh}} (\D ^{\dag} _{\X,\Q})$.
We put $\mathrm{Car} (\E ) := \cup _n \mathrm{Car} (\mathcal{H} ^{n} (\E))$. 
We say that $\E$ is Lagrangian if
for any integer $n$, the characteristic variety of $\mathcal{H} ^{n} (\E)$ is Lagrangian, i.e. 
if $\mathrm{Car} (\E ) $ is Lagrangian. 

\end{enumerate}
\end{dfn}

\begin{ex}
For instance, 
let $\iota \colon Z \hookrightarrow X$ be closed immersion of smooth $k$-varieties. 
Then 
$T ^{*} _Z X$ is Lagrangian. 
Indeed, since this local, from 
\cite[II.4.10]{sga1}, 
we can suppose there exist local coordinates $t _1, \dots, t _d$ of $X$ 
such that $t' _1, \dots, t '_r $, the global sections of $Z$ induced by $t _1,\dots, t _r$ via $\iota$, 
are local coordinates of $Z$ and such that $Z = V ( t _{r+1},\dots, t _d)$.
We denote by 
$\xi _1, \dots, \xi _d$ the global sections of $T ^{*} X$ associated with
$\partial _1,\dots, \partial _d$, the derivation with respect to $t _1,\dots, t _d$,
by
$\xi '_1, \dots, \xi '_r$  the global sections of $T ^{*} Z$ associated with
$\partial '_1,\dots, \partial '_r$, the derivation with respect to $t '_1,\dots, t '_r$, 
$\overline{\xi} _1, \dots, \overline{\xi} _r$ the global sections of 
$Z \times _{X} T ^* X$ induced by 
$\xi _1, \dots, \xi _r$ via the closed immersion 
$Z \times _{X} T ^* X \hookrightarrow X \times _{X} T ^* X =T ^* X $. 
Since $T ^* _Z Z= V ( \xi ' _1, \dots, \xi '_r)$, 
since 
$\xi '_1, \dots, \xi '_r$ are sent to 
$\overline{\xi} _1, \dots, \overline{\xi} _r$ via 
$\rho _{\iota}\colon  Z \times _{X} T ^* X \to T ^{*} Z$, 
we get 
$T ^{*} _{Z} X := \rho _{\iota} ^{-1} ( T ^* _Z Z)=
V ( \overline{\xi} _1, \dots, \overline{\xi} _r)$.
Since 
$Z \times _{X} T ^* X = V ( t _{r+1}, \dots, t _n)$, 
viewing $T ^{*} _{Z} X$ as a closed subvariety of 
$T ^* X$, we get
$T ^{*} _{Z} X =  V ( t _{r+1}, \dots, t _n, \xi _1, \dots, \xi _r)$.
This becomes obvious that
$\alpha _X = \sum _{i =1} ^{d} \xi _i d t _i $
vanishes on 
$T ^{*} _{Z} X$. 
Since it is also pure of codimension $d$, we are done. 
\end{ex}

\begin{empt}
\label{comput5.2.2.1Beintro}
Let $X$ be an affine smooth variety over $k$ admitting local coordinates
$t _{1}, \dots , t _d$.
We denote by $X ^{(m)}$ the base change of $X$ by the $m$th power of Frobenius of $S:= \Spec k$,
by $F ^{m}\colon X \to X ^{(m)}$ the relative Frobenius morphism. 
From the equalities \cite[1.1.3.1, 2.2.4.(iii)]{Be1}, we compute that for any $j<m$, we have in $\D ^{(m)} _{X/S}$ the equality
$(\partial _i ^{< p^j> _{(m)}} ) ^{p}=0$.
From \cite[1.1.3.(ii), 2.2.4.(iii)]{Be1}, for any $k$ and $l$, we compute that there exists $u \in \Z _p ^{*}$ such that
$(\partial _i ^{< p^m k> _{(m)}} ) ^{l}=u \partial _i ^{< p^m kl> _{(m)}} $.
Let $\xi _{i} ^{(m)}$ be the class of $\partial _i ^{< p ^{m}> _{(m)}} = \partial _i ^{[p ^{m}] }$ 
in 
$(\gr \, \D ^{(m)} _{X/S} ) _{\mathrm{red}}$.
Hence, with the formula \cite[2.2.5.1]{Be1}, 
we check that 
$(\gr \, \D ^{(m)} _{X/S} ) _{\mathrm{red}} = \oplus _{\underline{k} \in \N ^{d}} (\xi _i ^{(m)}) ^{k}  \O _{X}$.

From \cite[5.2.2]{Beintro2}, the canonical morphism 
$(\gr \, \D ^{(m)} _{X/S} ) _{\mathrm{red}}\to F ^{m} \gr \, \D ^{(0)} _{X ^{(m)}/S}$
induced by the morphism
$\D ^{(m)} _{X/S} \to F ^{m*} \D ^{(0)} _{X ^{(m)}/S}$ of left $\D ^{(m)} _{X/S} $-modules is an isomorphism.
We remark that this is consequence of above computations and of  \cite[2.2.4.3]{Be2} which states that  
the image of 
$\partial _i ^{< k> _{(m)}}$ via $\D ^{(m)} _{X/S} \to F ^{m*} \D ^{(0)} _{X ^{(m)}/S}$
is 
$1 \otimes \partial _i ^{k /p ^{m}}$ if $p ^{m}$ divides $k$ and 
otherwise is $0$
(and then $\xi _{i} ^{(m)}$ is
sent to $1 \otimes \xi _i$, where $\xi _{i} $ is the class of $\partial _i $ 
in 
$\gr \, \D ^{(0)} _{X ^{(m)}/S}$).

\end{empt}

\begin{empt}
[Local coordinates]
Let $\X$ be an integral separated smooth formal $\V$-scheme, $\ZZ$ be a strict normal crossing divisor of $\X$,
$\ZZ _1, \dots , \ZZ _r$ be the irreducible components of $\ZZ$.
We put $\X ^\# :=(\X, \ZZ) $ the corresponding smooth log formal $\V$-scheme.
We have defined in 
\cite[1.1]{caro_log-iso-hol}
the sheaf of differential operators of finite order of level $m$ on $\X ^\#$
and denoted by 
$\D ^{(m)} _{\X ^\#}$. 
Taking the $p$-adic completion and next taking
the inductive limit on the level, 
we get the sheaf of differential operators on $\X ^\#$ denoted by
$\D ^{\dag} _{\X ^\#}:= 
\underrightarrow{\lim} _m\smash{\widehat{\D}} ^{(m)} _{\X ^\#}$.
We have the following description in local coordinates
of $\D ^{(m)} _{\X ^\#}$. 
Suppose  $\X$ is affine with local coordinates
$t _{1}, \dots , t _d$ such that $\ZZ _i = V (t _i)$ for $i=1,\dots, r$. 
Then 
$\D ^{(m)} _{\X ^\#}$ 
(resp $\D ^{(m)} _{\X}$ is a free $\O _{\X}$-module) 
is a free $\O _{\X}$-module, 
with basis (naturally constructed in the sense that it 
only depends on the choice of $t _1,\dots, t _d$) 
denoted by 
$\underline{\partial} _\# ^{<\underline{k}> _{(m)}}$
(resp. $\underline{\partial}^{<\underline{k}> _{(m)}}$), 
where $\underline{k}= (k _1, \dots, k _d)\in \N ^d$.  
and when $\underline{k}=\underline{\epsilon} _i= (0,\dots, 0,1,0,\dots, 0)$ with the $1$ at the $i$th place,
we put $\partial _{\# i}:= \underline{\partial} _\# ^{<\epsilon _i> _{(m)}}$
(resp. $\partial _{i}:= \underline{\partial} ^{<\epsilon _i> _{(m)}}$).
$\partial _{\#1},\dots, \partial _{\#d}$ 
(resp. $\partial _{1},\dots, \partial _{d}$) are the logarithmic derivation with respect to 
$t _1,\dots, t _d$ 
(resp. derivation with respect to $t _1,\dots, t _d$). 
We have the inclusion 
$\D ^{(m)} _{\X ^\#} \subset \D ^{(m)} _{\X }$
and 
the relation
$\underline{\partial} _\# ^{<\underline{k}> _{(m)}}
=
\underline{t} ^{\underline{k}}\underline{\partial}  ^{<\underline{k}> _{(m)}}$,
where
$\underline{t} ^{\underline{k}}:= 
t _{1} ^{k _1}\cdots t _{d} ^{k _d}$.
In the ring $\D ^{(m)} _{\X ,\Q} \supset \D ^{(m)} _{\X }$, we have the relation
$$\underline{\partial}  ^{<\underline{k}> _{(m)}}
=
\frac{[\underline{k}/p ^m]!}{\underline{k}!}\underline{\partial} ^{\underline{k}},$$
where 
$\underline{k}!:= k _1 !\cdots, k _d$, 
$[\underline{k}/p ^m]!:= [k _1/p ^m] ! \cdots, [k _d/p ^m] ! $,
$\underline{\partial} ^{\underline{k}}:= 
\partial _1^{k _1}\cdots \partial _d ^{k _d}$, 
which explain the notation. 

\end{empt}

\begin{prop}
\label{lagrangianityLogextend}
Let $\X$ be an integral separated smooth formal $\V$-scheme, $\ZZ$ be a strict normal crossing divisor of $\X$. 
We put $\X ^\# :=(\X, \ZZ) $ the corresponding smooth log formal $\V$-scheme.
Let $\G$ be a coherent $\D ^{\dag} _{\X ^\#,\Q}$-module which is also 
a locally projective $\O _{\X,\Q}$-module of finite type (i.e., following \cite[4.9 and 4.15]{caro_log-iso-hol},
 $\G$ is a convergent isocrystal on $\X ^{\#}$).
Let $\E := (\hdag Z) (\G)$ be the induced isocrystal on $X \setminus Z$ overconvergent along $Z$.
We suppose that $\G$ is endowed with a Frobenius structure. 
Let $\ZZ _1, \dots , \ZZ _r$ be the irreducible components of $\ZZ$.
For any subset $I$ of $\{ 1, \dots , r\}$ (including $I = \emptyset$), we put $\ZZ _{I} : =\cap _{i \in I} \ZZ _{i}$
(in particular, we have $\ZZ _{\emptyset}=\X$).
Then we have the inclusion 
\begin{equation}
\label{lagrangianityLogextend-incl}
|\mathrm{Car} (\E )|
\subset \cup _{I \subset \{ 1, \dots , r\}}T ^{*} _{Z _I} X,
\end{equation}
where $T ^{*} _{Z _I} X$
is the standard notation (see \ref{stab-T_X X}).
In particular, with the remark \ref{Def-dev-Lagrangianity} and since we know that $\E$ is holonomic (see \cite{caro-Tsuzuki}), 
this implies that 
$|\mathrm{Car} (\E )|$ is Lagrangian. 
\end{prop}

\begin{proof}
This is local so we can suppose $\X$ affine with local coordinates
$t _{1}, \dots , t _d$ such that $\ZZ _i = V (t _i)$ for $i=1,\dots, r$. 
We proceed by induction on $r$. 
For $r =0$ (i.e. $\ZZ$ is empty), this is already known (see the example after \cite[5.2.7]{Beintro2}). Suppose now $r \geq 1$.
From \cite[4.12]{caro_log-iso-hol}, there exists 
a coherent $\widehat{\D} ^{(0)} _{\X ^{\#}}$-module 
$\G ^{(0)}$ which is also $\O _{\X}$-coherent and such that 
$\G ^{(0)} _\Q \riso \G$. 
We can suppose that $\G ^{(0)}$ has no $p$-torsion (indeed, from \cite[3.4.4]{Be1}, the subsheaf of $\G ^{(0)}$ of $p$-torsion elements is 
a coherent $\widehat{\D} ^{(0)} _{\X ^{\#}}$-module and also a coherent $\O _{\X}$-module).
With the notation 
\cite[5.1]{caro_log-iso-hol}, 
$\G (\ZZ)$ is also a coherent $\D ^{\dag} _{\X ^\#,\Q}$-module and 
a locally projective $\O _{\X,\Q}$-module of finite type.
Hence, from \cite[4.14]{caro_log-iso-hol}, we get 
$\widehat{\D} ^{(m)} _{\X ^{\#},\Q} \otimes _{\widehat{\D} ^{(0)} _{\X ^{\#},\Q}} \G  (\ZZ)
\riso \G  (\ZZ)$.
We denote by 
$\H ^{(m)}$
the quotient of 
$\widehat{\D} ^{(m)} _{\X ^{\#}} \otimes _{\widehat{\D} ^{(0)} _{\X ^{\#}}} (\G ^{(0)} (\ZZ))$
by its $p$-torsion part.
The latter isomorphism implies that
$\H ^{(m)} _\Q \riso \G  (\ZZ)$.
By using \cite[3.4.5]{Be1}
and \cite[4.12]{caro_log-iso-hol}, it follows that
$\H ^{(m)}$ is isogeneous to 
a coherent $\widehat{\D} ^{(m)} _{\X ^{\#}}$-module 
which is also $\O _{\X}$-coherent.
Since $\Gamma (\X, \O _{\X})$ is noetherian and 
$\H ^{(m)}$ has no $p$-torsion, 
we get that $\Gamma ( \X, \H ^{(m)})$ is a $\Gamma (\X, \O _{\X})$-module of finite type. 
Since $\H ^{(m)}$ is a coherent $\widehat{\D} ^{(m)} _{\X ^{\#}}$-module, 
this yields that $\H ^{(m)}$ is also $\O _{\X}$-coherent
(this is a log-variation of \cite[2.2.13]{caro_courbe-nouveau} and its check is identical).

\bigskip

We denote by $\M ^{(m)}:= \widehat{\D} ^{(m)} _{\X} \otimes _{\widehat{\D} ^{(m)} _{\X ^{\#}}} \H ^{(m)}$
and by $\NN ^{(m)}$ the quotient of   
$\M ^{(m)}$
by its $p$-torsion part.
We put 
$\overline{\H} ^{(m)} : =\H ^{(m)}/ \pi \H ^{(m)}$,
$\overline{\M} ^{(m)}: =\M ^{(m)}/ \pi \M ^{(m)}$
and $\overline{\NN} ^{(m)}: =\NN ^{(m)}/\pi \NN ^{(m)}$.	
Since we have the epimorphism
$\overline{\M} ^{(m)} \twoheadrightarrow \overline{\NN} ^{(m)}$, 
from \cite[5.2.4.(i)]{Beintro2} and its notation, we get 
$\mathrm{Car} ^{(m)} (\overline{\NN} ^{(m)}) \subset \mathrm{Car} ^{(m)} (\overline{\M} ^{(m)})$.
From \cite[4.14]{caro_log-iso-hol}, we get the second isomorphism
$\D ^{\dag} _{\X ,\Q}  \otimes _{\widehat{\D} ^{(m)} _{\X ,\Q}} \NN ^{(m)} _\Q
\riso 
\D ^{\dag} _{\X ,\Q} \otimes _{\widehat{\D} ^{(m)} _{\X ^\#,\Q}} \G (\ZZ)
\riso 
\D ^{\dag} _{\X ,\Q} \otimes _{\D ^{\dag} _{\X ^\#,\Q}} \G (\ZZ)$.
Since $\G$ has a Frobenius structure, it follows from 
\cite[5.24.(ii)]{caro_log-iso-hol}
and
\cite[2.2.9]{caro-Tsuzuki}
that we have the isomorphism 
$\D ^{\dag} _{\X ,\Q} \otimes _{\D ^{\dag} _{\X ^\#,\Q}} \G (\ZZ)
\riso \E$. 
Hence, 
$\D ^{\dag} _{\X ,\Q}  \otimes _{\widehat{\D} ^{(m)} _{\X ,\Q}} \NN ^{(m)} _\Q
\riso \E$. 
By definition of 
$\mathrm{Car} (\E)$
(see \cite[5.2.7]{Beintro2}), 
by using \cite[3.6.2.(i)]{Be1} this implies that 
for $m$ large enough we have the equality
$\mathrm{Car} (\E)= \mathrm{Car} ^{(m)} (\overline{\NN} ^{(m)}) $
(modulo the homeomorphism
between $T ^{*(m)} X$ and $T ^{*} X$ of \cite[5.2.2.1]{Beintro2}).

Let $\overline{\M} ^{(m)} _n $ be the image of 
$\D ^{(m)} _{X, n} \times \overline{\H} ^{(m)}\to \overline{\M} ^{(m)} $.
Since $\overline{\H} ^{(m)}$ is 
$\O _{X}$-coherent, 
we check that $\overline{\M} ^{(m)} _n $ is a good filtration of $\overline{\M} ^{(m)}$ (see \cite[5.2.3]{Beintro2}). 
By definition, this implies
$\mathrm{Car} ^{(m)} (\overline{\M} ^{(m)})
=
\mathrm{Supp} (\oplus _{n\in \N}
(\overline{\M} ^{(m)} _n  / \overline{\M} ^{(m)} _{n-1}))$. 
For $d \geq i >r$, we remark that
$\partial _i ^{< p ^{m}> _{(m)}} \in \D ^{(m)} _{X ^\#}$.
Hence, we get in $T ^{*(m)} X = \Spec (\gr \, \D ^{(m)} _{X/S} ) _{\mathrm{red}}$ the inclusion 
$\mathrm{Car} ^{(m)} (\overline{\M} ^{(m)}) \subset \cap _{i=r+1} ^{d} V ( \xi _{i} ^{(m)})$ (see the notation of \ref{comput5.2.2.1Beintro}).
Since,  modulo the homeomorphism between $T ^{*(m)} X$ and $T ^{*} X$ of \cite[5.2.2.1]{Beintro2}, 
the closed variety $\cap _{i=r+1} ^{d} V ( \xi _{i} ^{(m)})$ corresponds to 
$\cap _{i=r+1} ^{d} V ( \xi _{i} )$ (see the description given in \ref{comput5.2.2.1Beintro}
of the homeomorphism \cite[5.2.2.1]{Beintro2}), 
then we get 
$\mathrm{Car}  (\E) \subset \cap _{i=r+1} ^{d} V ( \xi _{i} )$.

Let $\ZZ _{min} := \cap _{i=1} ^{r}\ZZ _i$.
Then, we get $\iota ^{-1} (\mathrm{Car}  (\E))
\subset \iota ^{-1} ( \cap _{i=r+1} ^{d} V ( \xi _{i}))
= T ^{*} _{Z _{min}} X$, 
where 
$\iota \colon Z _{min} \times _{X} T ^{*}X \hookrightarrow T ^{*}X$ is the canonical immersion.

For any $i = 1,\dots, r$, we put 
$\X  _i := \X \setminus  \ZZ _i$.
From the induction hypothesis, 
we get the first inclusion
$\mathrm{Car}  (\E |\X _i) 
\subset 
\cup _{l _i \in I _i}T ^{*} _{X _i \cap Z _{I _i}} X _i
=
X _i \times _{X}\cup _{l _i \in I _i}T ^{*} _{Z _{I _i}} X$,
where the union runs through
subsets $I _i$ of $\{ 1, \dots , r\}$ which do not contain $i$.
Since $(\X _i) $ is a open covering of $\X \setminus Z _{min}$ and 
since 
$\iota ^{-1} (\mathrm{Car}  (\E))
\subset T ^{*} _{\ZZ _{min}} X$, we conclude. 

\end{proof}

\begin{rem}
Let $\E ' \to \E \to \E '' \to \E '[1]$ be an exact triangle of $F \text{-}D ^{\mathrm{b}} _{\mathrm{coh}} (\D ^{\dag} _{\X,\Q})$.
Using \cite[5.2.4.(i) and 5.2.7]{Beintro2}, 
we get the equality 
$\mathrm{Car} (\E ) = \mathrm{Car} (\E' ) \cup \mathrm{Car} (\E'')$.
This yields that  $\E'$ and $\E''$ are Lagrangian if and only if $\E$ is Lagrangian.
Hence, to check the Lagrangianity of overholonomic $F$-complexes, we reduce by devissage to the case of overcoherent $F$-isocrystals (see \cite{caro_devissge_surcoh}).
But, this is probably wrong that any overcoherent $F$-isocrystals are Lagrangian. 
\end{rem}

\small
\bibliographystyle{alpha}
\def\cprime{$'$}
\providecommand{\bysame}{\leavevmode ---\ }
\providecommand{\og}{``}
\providecommand{\fg}{''}
\providecommand{\smfandname}{et}
\providecommand{\smfedsname}{\'eds.}
\providecommand{\smfedname}{\'ed.}
\providecommand{\smfmastersthesisname}{M\'emoire}
\providecommand{\smfphdthesisname}{Th\`ese}

\bigskip
\noindent Daniel Caro\\
Laboratoire de Mathématiques Nicolas Oresme\\
Université de Caen
Campus 2\\
14032 Caen Cedex\\
France.\\
email: daniel.caro@unicaen.fr

\end{document}